\documentclass[reqno]{amsart}
\usepackage{graphicx} 
\usepackage{amsfonts,amsmath,amsthm,amssymb, 
mathtools,hyperref,verbatim}
\usepackage{geometry} 
\usepackage{pgfplots}
\usepackage[backend=biber,style=numeric,sorting=nyt,giveninits=true]{biblatex}
\usepackage{hyperref}
\usepackage{url}

\pgfplotsset{width=10cm,compat=1.9}
\title{Additional Constructions of Sequences of Alternating Sum and Difference Dominated Sets}
\author{Yorick Herrmann}
\email{\textcolor{blue}{\href{mailto:yherrman@uci.edu}{yherrman@uci.edu}}}

\author{Connor Hill}
\email{\textcolor{blue}{\href{mailto:hill.connor.03@gmail.com}{hill.connor.03@gmail.com}}}

\author{Merlin Phillips}
\email{\textcolor{blue}{\href{mailto:merlin.phillips216@gmail.com}{merlin.phillips216@gmail.com}}}
\author{Daniel Flores}
\email{\textcolor{blue}{\href{mailto:flore205@purdue.edu}{flore205@purdue.edu}}}

\author[Miller]{Steven J. Miller}
\email{\textcolor{blue}{\href{mailto:sjm1@williams.edu}{sjm1@williams.edu}}
\textcolor{blue}{\href{Steven.Miller.MC.96@aya.yale.edu}{Steven.Miller.MC.96@aya.yale.edu}}}
\address{Department of Mathematics, Williams College,
Williamstown, MA 01267, USA}

\author{Steven Senger}
\email{\textcolor{blue}{\href{mailto:StevenSenger@missouristate.edu}{StevenSenger@missouristate.edu}}}
\address{Department of Mathematics, Missouri State University,
Springfield, MO 65897, USA}

\date{\today}

\newcounter{master}[section]

\newtheorem{thm}[master]{Theorem}
\newtheorem{lem}[master]{Lemma}

\newtheorem*{rem}{Remark}

\makeatletter
\let\c@table\c@master

\makeatother

\newcommand{\sss}{\big|A+A\big|}
\newcommand{\dss}{\big|A-A\big|}

\begin{document}


\begin{abstract}
A More Sums Than Differences (MSTD) set is a finite set of integers $A$
where the cardinality of its sumset, $A+A$, is greater than the
cardinality of its difference set, $A-A$. We address a problem posed by Samuel Allen Alexander that asks whether there exists an infinite sequence of sets alternating
between being MSTD and More Differences Than Sums (MDTS), where each set properly contains the previous. While a companion paper resolved this using `filling in' techniques, we solve the more challenging `non-filling-in' version, where any missing integer between a set's minimum and maximum elements remains missing in all subsequent sets.
\end{abstract}

\maketitle

\section{Introduction}

The sumset $A+A$ and the difference set $A-A$ of a finite set of integers $A$ are defined to be the set of all possible sums or differences between elements of the set $A$. More formally, $A+A \coloneqq \{a+b: a,b \in A\}$ and $A-A \coloneqq \{a-b: a,b \in A\}$, with their respective cardinalities denoted $\sss$ and $\dss$. A \textit{More Sums Than Differences} (MSTD) or \textit{sum-dominated} set is a set $A$ for which $\sss  > \dss$. Similarly, a \textit{More Differences Than Sums} (MDTS) or \textit{difference-dominated} set is a set $A$ for which $\dss  > \sss$. We should generally expect a set $A$ to be difference-dominated since addition is commutative while subtraction is not.  

We see early examples of sum-dominated sets discovered by Conway $\big(\{0,2,3,4,7,11,12,14\}\big)$, Marica $\big(\{1, 2, 3, 5, 8, 9, 13, 15, 16\}\big)$ \cite{MaricaConwayConjecture}, and Freiman-Pigarev
\footnote{$\{0, 1, 2, 4, 5,9, 12, 13, 14, 16, 17, 21, 24, 25, 26, 28, 29\}$}\cite{FreimanPigarev1973}. The set found by Conway was later shown by Hegarty \cite{HegartyMinSum} to be the smallest possible MSTD set with respect to cardinality and diameter. The problem was formalized by Nathanson, who noted that MSTD sets should be rare and proved methods of their construction along with several properties of such sets \cite{nathanson2006}. Despite these sets being rare, Martin and O’Bryant \cite{MartinOBryant} proved that a positive percentage of subsets of $\{1,2,\dots,n\}$ are sum dominated as $n\to\infty$. Note this model (all subsets being equally likely to be chosen) is equivalent to each element of $\{1,2,\dots,n\}$ being in our set with probability $1/2$; Hegarty and Miller \cite{HM} proved that if each element is chosen uniformly and independently with probability $p(n)$, then if $p(n) \to 0$, almost all sets are difference dominated.
Work has also been done to find explicit constructions of generalized MSTD sets \cite{miller2008,ExplicitLargeFamilies,nathanson2017} and generalizations to groups \cite{Zhao2010}. Sumsets and difference sets are also present in many problems in number theory.  

At the recent 2025 CANT (Combinatorial and Additive Number Theory Conference), Samuel Allen Alexander posed the problem of finding an infinite sequence of sets with $A_{i-1} \subset A_i$ that alternate being MSTD and MDTS. In a companion paper \cite{AltSetsPt1}, we addressed this problem at first by  `filling in' sets in order to generate subsequent sets in the sequence. Filling in a set $A_i$ refers to the process of adding elements of $[a,b] \setminus A_i$ to $A_i$. However, filling in can distort the structure of earlier sets in the sequence, resulting in relatively trivial sequences with rapidly growing sets. Constructions with slower growth rates are more desirable, as they extend more naturally to algebraic settings where the construction of long alternating chains is constrained by the underlying structure, such as finite fields or abelian groups. 

In \cite{AltSetsPt1}, we also introduced a method that avoids filling in altogether and produced a sequence which retained the structural properties of the initial set $A_1$. In this paper, we build on this work by providing three additional methods which prohibit the use of filling in. Before introducing these methods, we first establish several known properties of both general and MSTD sets.  

For any set of integers $A$, we define the dilation $x \cdot A \coloneqq \{xa:a \in A\}$. For all integers $x,y$ such that $x \ne  0$, the set
\begin{align}B \ = \ x \cdot A+\{y\} \ = \ \{xa+y:a\in A\}\end{align} 
satisfies $\big|B+B\big| = \sss$ and $\big|B-B\big| = \dss$ \cite{nathanson2006}. This property ensures that any method developed on a restricted subset of the integers can be extended, via dilation and translation, to the full domain of integers.

Furthermore, through dilation and translation, we have that any non-trivial $k$-term arithmetic progression can be obtained from the interval 
\[I \ = \ \{1,2,\ldots,k\} \ := \ [1,k].\] 
It is also evident that \begin{align}
\notag I+I \ &= \ [2,2k]\ \\ 
I-I \ &= \ [1-k,k-1],
\end{align} so such sets have the same number of sums and differences. For instance, if $A=\{0,1,2\}$, then \begin{align}
   \notag A+A \ &= \ \{0,1,2,3,4\} \\
   \notag A-A \ &= \ \{-2,-1,0,1,2\} \\
   \sss \ &= \ \dss \ = \ 5.
\end{align}

For a set of integers $A$, if there exists an $a^* \in \mathbb{Z}$ such that $a^*-A = A$, then we say that $A$ is symmetric with respect to $a^*$, and we have that $\sss = \dss$, and if $A$ is an arithmetic progression, we have that $A$ is symmetric to $a^*=\min(A)+\max(A)$ \cite{nathanson2006}.

Many methods that are used to construct MSTD sets make use of the symmetry property \cite{nathanson2006}. For example, the set $A = \{0,2,3,7,11,12,14\}$ is symmetric with respect to $14$, so $\sss = \dss$. However, adding $4$ to this set forms Conway's set, which is MSTD. We also use this symmetry property in one method presented in the paper.  

Adding to the work of our previous paper \cite{AltSetsPt1}, we provide additional non-trivial non-filling in methods to obtain the desired sequences. Each method achieves a more efficient growth rate than the filling in methods presented in \cite{AltSetsPt1}. Theorem \ref{NFIMethod1Theorem} provides a method that generates the sequence by adjoining elements chosen to preserve the sum-difference imbalance at each step.
Theorem \ref{NFIMethod3Theorem} refines an approach of Nathanson by imposing additional constraints to achieve an even slower growth rate in the resulting sequence. Ideas from the proof of Theorem \ref{NFIMethod3Theorem} are used to construct a single sequence which is our slowest-growing result.

\section{Non-Filling in Methods}
We formalize the `non-filling in' constraint on the alternating sequence we seek to generate. For any $A_k$, let $a = \min A_k$ and $b = \max A_k$. Then, if an $x \in [a,b]$ is not in $A_k$, we require that $x \notin A_n$ for all $n > k$. In other words, this restriction prohibits filling in when generating the sequence. We provide several non-filling in methods below.

\subsection{Non-Filling in Method 1}

First, we construct a non-filling in method that essentially `copies' the initial set $A_1$ around a suitably chosen modulus 
$n$, ensuring that the initial structure of $A_1$ is maintained in later sets in the sequence. Theorem \ref{NFIMethod1Theorem} summarizes this method.

\begin{thm} \label{NFIMethod1Theorem}
Let $A_1$ be a MSTD set with $0\in A_1$, such that there is an integer $n > \max A_1$ satisfying the following conditions.
\begin{enumerate}
    \item The number of unique elements in $A_1+A_1$ modulo $n$ is the same as the number of unique elements of $A_1-A_1$ modulo $n$. 
    \item Let $x$ be the number of $a\in A_1$ such that $n+a\notin A_1+A_1$, and let $y$ be the number of $b\in A_1$ such that $n-b\notin A_1-A_1$. We require that $2y-x-1 > \big|A_1+A_1\big|-\big|A_1-A_1\big|$.
\end{enumerate}
For $l \geq 1$, we define \begin{align}
    \notag A_{2l} \ &\coloneqq \ A_{2l-1} \ \cup \ \{ln\} \\
    \notag A_{2l+1} \ &\coloneqq \ A_{2l-1} \ \cup \ (A_1+ln).
\end{align}

Then $A_{2l}$ is MDTS, $A_{2l+1}$ is MSTD, and $A_1 \subset \cdots \subset A_{2l-1} \subset A_{2l} \subset A_{2l+1} \subset \cdots$ forms the desired alternating sequence.

\end{thm}

\begin{proof}   

First, set $A_2 = A_1\cup\{n\}$. We show that $A_2$ is difference-dominated.

The additional sums created by $n$ are $n+n$ and $\{n+a:a\in A_1\}$. In the latter case, a new sum is realized whenever $n+a\notin A_1+A_1$. Thus
\begin{equation}
    \big|A_2+A_2\big| \ = \ \big|A_1+A_1\big|+x+1.
\end{equation}
\indent The additional differences created by $n$ are $\pm\{n-a,a\in A_1\}$; since a difference is realized if and only if its negative is realized, we only count differences of the form $(n-a)$ and double the result. This gives a new difference whenever $n-a\notin A_1-A_1$. Thus
\begin{equation}
    \big|A_2-A_2\big|\ = \ \big|A_1+A_1\big|+2y,
\end{equation} 
and we have
\begin{align}
    \notag\big|A_2-A_2\big|-\big|A_2+A_2\big| \ &= \ \big|A_1-A_1\big|-\big|A_1+A_1\big|+2y-x-1 \\
    &= \ 2y-x-1-\big(\big|A_1+A_1\big|-\big|A_1-A_1\big|\big).
\end{align}
Using the second assumption, we get $\big|A_2-A_2\big|-\big|A_2+A_2\big| > 0$.

Let $A_3 = A_1+\{0,n\}$ \cite{MartinOBryant}. We show that $A_3$ is sum-dominated. For $A_3+A_3$, there are two types of additional sums to consider.
\begin{enumerate}
    \item[] \textbf{Case 1:} $\{a+n+b+n:a,b\in A_1\}$. The largest element of $A_1+A_1$ is $2\cdot\max A_1 < 2n$, so all sums of this type are unaccounted for. Furthermore, every new sum corresponds to a unique value of $a+b$, so there are $\sss$ new sums.
    \item[] \textbf{Case 2:} $\{a+n+b:a,b\in A_1\}$. These give new sums unless there are $a',b'\in A_1$ with $a'+b'-n = a+b$ or $a'+b'+n = a+b$ (as they occur in $A_1+A_1$ or in Case 1, respectively). As such, the new sums are the elements of $A_1+A_1$ that occur exactly once modulo $n$.
\end{enumerate}
Note that since $2n  > \max(A_1+A_1)$, for each element that appears in $A_1+A_1$ modulo $n$, there are either one or two sums in $A_1+A_1$ that are congruent to that element modulo $n$. If two, the element is double-counted by $\big|A_1+A_1\big|$ in Case 1. If one, the element is counted once by Case 1 and once again by Case 2. Thus the amount of new sums is twice the number of unique elements of $A_1+A_1$ modulo $n$.

Next, for $A_3-A_3$, there are two types of additional differences to consider.
\begin{enumerate}
    \item[]\textbf{Case 1:} $\{a+n-b-n:a,b\in A_1\}$. This is just $A_1-A_1$, so there are no new differences.
    \item[]\textbf{Case 2:} $\pm\{a+n-b:a,b\in A_1\}$. Again, we only count the new positive differences and then double the total. A positive difference is new unless there are $a',b'\in A_1$ with $a'-b'-n = a-b$, so the number of new positive differences is the number of unique differences modulo $n$.
    \begin{rem}
        Unlike Case 2 of the sum counting, we don't count only the differences which occur exactly once modulo $n$. If $a'-b'-n = a-b$, we reject $a-b$ but we already have $a'-b'$.
   \end{rem}
\end{enumerate}
Thus, the amount of new differences is twice the number of unique elements of $A_1-A_1$ modulo $n$.

Using the first assumption, we note that the number of new sums in $A_3+A_3$ is the same as the number of new differences in $A_3-A_3$. Thus
\begin{equation}
    \big|A_3+A_3\big|-\big|A_3-A_3\big| \ = \ \big|A_1+A_1\big|-\big|A_1-A_1\big|.
\end{equation}

Since $A_1$ is sum-dominated, we have that $A_3$ is sum-dominated as well. Additionally, we have
\begin{equation}
    A_1\subset A_2\subset A_3.
\end{equation}

For $l\ge2$, we define $A_{2l+1}$ as in the statement of Theorem \ref{NFIMethod1Theorem}, i.e.,
\begin{equation}
    A_{2l+1} \ = \ A_{2l-1}\cup(A_1+ln).
\end{equation}
The additional sums in $A_{2l+1}+A_{2l+1}$ that are not in $A_{2l-1}+A_{2l-1}$ are $\{a+ln+b+kn\}$, where $0 \leq  k \leq l, k\in \mathbb{Z}$. For $k \ne l,l-1$, the sum
\begin{equation}
    a+ln+b+kn \ = \ a+(l-1)n+b+(k+1)n
\end{equation}
is not a new sum since $A_1+\{(l-1)n\} \subset A_{2l-1}$ and $A_1+\{(k+1)n\} \subset A_{2l-1}$. We then consider the following types of sums.
\begin{enumerate}
    \item[]\textbf{Case 1:} $\{a+b+2ln\}$. The largest element of $A_{2l-1}+A_{2l-1}$ is $2\cdot\max A_1+2(l-1)n < 2ln$, so these sums are all unaccounted for. Therefore, there are $\big|A_1+A_1\big|$ new sums.
    \item[]\textbf{Case 2:} $\{a+b+(2l-1)n\}$. These give new sums unless there are $a',b'\in A$ with $a'+b'-n = a+b$ or $a'+b'+n = a+b$ (as they occur in $A_{2l-1}+A_{2l-1}$ or in Case 1, respectively). The number of new sums is therefore equal to the number of elements of $A_1+A_1$ that occur exactly once modulo $n$.
\end{enumerate}
Thus, as with $A_3+A_3$, the amount of new sums is twice the number of unique elements of $A_1+A_1$ modulo $n$.

The additional differences in $A_{2l+1}-A_{2l+1}$ are $\pm\{a+ln-b-kn\}$, where $0 \leq k \leq l$, $k \in \mathbb{Z}$. For $k \ne 0$,
\begin{equation}
    a+ln-b-kn \ = \ a+(l-k)n-b
\end{equation}
is not a new difference since $(A_1-A_1) + \{(l-k)n\} \subset A_{2l-1} -A_{2l-1}$. The $\{a+ln-b\}$ give new differences unless there are $a',b'\in A_1$ with $a'-b'-n = a-b$. Thus, as with $A_3-A_3$, the number of new differences is twice the number of unique elements of $A_1-A_1$ modulo $n$. Therefore
\begin{align}
    \notag\big|A_{2l+1}+A_{2l+1}\big|-\big|A_{2l+1}-A_{2l+1}\big| \ &= \ \big|A_{2l-1}+A_{2l-1}\big|-\big|A_{2l-1}-A_{2l-1}\big|\\
    \ &= \ \big|A_1+A_1\big|-\big|A_1-A_1\big|,  \label{Sec3Method1EqualDiffBWsSumAndDiff}
\end{align}
and $A_{2l+1}$ is sum-dominated. \\

 We now define $A_{2l}$ as in the statement of Theorem \ref{NFIMethod1Theorem}, i.e.,
 \begin{equation}
    A_{2l} \ = \ A_{2l-1}\cup\{ln\}.
 \end{equation}
 The additional sums in $A_{2l}+A_{2l}$ are $2ln$ and the set of $\{ln+a+kn\}$, where $a \in A_1$ and $0 \leq k \leq l-1$, $k \in \mathbb{Z}$. For $k \ne l-1$,
 \begin{equation}
    a+ln+kn \ = \ a+(l-1)n+(k+1)n
 \end{equation}
 is not a new sum since $0  \leq k+1 \leq l-1$. When $k = l-1$, we have that $a+(2l-1)n$ is a new sum whenever $a+(2l-1)n \notin A_1+A_1+\{(2l-2)n\}$, since $A_1+A_1+\{(2l-2)n\}$ is the only possible subset of $A_{2l-1} +A_{2l-1}$ that $a+(2l-1)n$ can be a part of. The total number of $a$ that satisfy this is $x$ since $a+(2l-1)n \notin A_1+A_1+\{(2l-2)n\} \iff n + a \notin A_1+A_1$. Thus
 \begin{equation}
     \big|A_{2l}+A_{2l}\big| \ = \ \big|A_{2l-1}+A_{2l-1}\big|+x+1. \label{+x+1}
 \end{equation}
The additional differences in $A_{2l}-A_{2l}$ are $\pm\{ln-b-kn\}$, where $b \in A_1$ and $0 \leq k \leq l-1$, $k \in \mathbb{Z}$. For $k \ne 0$,
\begin{equation}
    ln-b-kn \ = \ (l-k)n-b
\end{equation}
is not a new difference since $0 \leq l-k \leq l-1$. When $k = 0$, $ln-b$ is a new (positive) difference whenever $ln-b \notin \{(l-1)n\}+A_1-A_1$, since $\{(l-1)n\}+A_1-A_1$ is the only possible subset of $A_{2l-1} - A_{2l-1}$ that $ln-b$ can be a part of. The total number of $b$ that satisfy this is $y$ since $ln-b \notin \{(l-1)n\}+A_1-A_1 \iff n - b \notin A_1-A_1$. Thus 
\begin{equation}
    \big|A_{2l}-A_{2l}\big| \ = \ \big|A_{2l-1}-A_{2l-1}\big|+2y,
\end{equation} and combining this with Equations \eqref{Sec3Method1EqualDiffBWsSumAndDiff} and \eqref{+x+1}, we get 
\begin{align}
    \notag\big|A_{2l}-A_{2l}\big|-\big|A_{2l}+A_{2l}\big| &= \big|A_{2l-1}-A_{2l-1}\big|-\big|A_{2l-1}+A_{2l-1}\big|+2y-x-1\\
    \notag &= \big|A_1-A_1\big|-\big|A_1+A_1\big|+2y-x-1\\
     &= 2y-x-1-(\big|A_1+A_1\big|-\big|A_1-A_1\big|).
\end{align}
Hence, from our second assumption, $A_{2l}$ is difference-dominated.
Finally, we note that
\begin{equation}
    A_{2l-1}\subset A_{2l}\subset A_{2l+1}.
\end{equation}
\indent Thus, the sequence $A_1\subset A_2 \subset \cdots$ alternates between being MSTD and MDTS. 

\end{proof}

We briefly verify that the `no filling in' constraint is met. For $l \ge 1$,
\begin{align}
    \notag \max A_{2l-1} \ &< \ ln\\
    A_{2l}\setminus A_{2l-1} \ &= \ \{ln\},
\end{align}
and
\begin{align}
    \notag \max A_{2l}& \ = \ ln\\
    A_{2l+1}\setminus A_{2l}\subset[&ln,(l+1)n-1].
\end{align}

As such, this sequence satisfies the constraint. 

We give an example of a sequence generated by this method. Let $A_1 = \{0,2,3,4,7,11,12,14\}$, the Conway set. For $n = 17$, the sets $A_1+A_1$ and $A_1-A_1$ both have 17 unique elements modulo $n$, so the first condition of Theorem \ref{NFIMethod1Theorem} is met. Furthermore, it can be computationally verified that for $n = 17$, we have $x = 3$, $y = 4$, $\big|A_1+A_1\big| = 26$, and $\big|A_1-A_1\big| = 25$. Therefore, $2y-x-1 > \big|A_1+A_1\big| -\big|A_1-A_1\big|$, and the second condition of Theorem \ref{NFIMethod1Theorem} is met. 

Applying the method, we get \begin{align}
    \notag A_2 \ &= \ A_1 \ \cup \{17\} \\
    \notag A_3 \ &= \ A_1 \ \cup \ \{a+17  \mid a \in A_1\} \\
    \notag A_4 \ &= \ A_3 \ \cup \{34\} \\
    A_5 \ &= \ A_4 \ \cup \ \{a+34  \mid a \in A_1\},
\end{align}
and so on. The cardinalities and diameters of the sets in this sequence are given in Table \ref{NFIM1Table}. The density of a set refers to that set's cardinality divided by its diameter. The D($A_i$)/D($A_{i-1}$) column measures how much larger a set's diameter is compared to the previous set in the sequence. Note that all entries in the density and growth rate columns are rounded to 3 decimal places, and this convention will be followed throughout the rest of the paper.

\begin{table}[h] 
\centering
\caption{Non-Filling in Method 1 Example Sequence}
\label{NFIM1Table}
\begin{tabular}{|c|c|c|c|c|c|c|c|}
\hline
Set & $\big|A_i+A_i\big|$ & $\big|A_i-A_i\big|$ & Cardinality & Diameter & $\big|A_i\big|/\big|A_{i-1}\big|$ & D($A_i$)/D($A_{i-1}$)/& Density \\
\hline
$A_1$ & 26 & 25 & 8 & 14 & N/A & N/A & 0.571\\
$A_2$ & 30 & 33 & 9 & 17 & 1.125 & 1.214 & 0.529 \\
$A_3$ & 60 & 59 & 16 & 31 & 1.778 & 1.824 & 0.516\\
$A_4$ & 64 & 67 & 17 & 34 & 1.063 & 1.097 & 0.500 \\
$A_5$ & 94 & 93 & 24 & 48 & 1.412 & 1.412 & 0.500 \\
$A_6$ & 98 & 101 & 25 & 51 & 1.042 & 1.063 & 0.490\\
$A_7$ & 128 & 127 & 32 & 65 & 1.280 & 1.275 & 0.492 \\
\vdots & \vdots & \vdots & \vdots & \vdots & \vdots & \vdots &\vdots\\
\hline
\end{tabular}\\
\raggedleft Limiting MSTD density: $0.471\ (8/17)$ \hspace{0.8cm}
\end{table}

As seen in the table, the diameter of each MSTD set in the chain is exactly $n = 17$ larger than the most recent MSTD set in the sequence. Furthermore, the cardinality of each MSTD set in the chain is exactly $\big|A_1\big| = 8$ larger than the most recent MSTD set in the sequence. As such, we conclude that the growth rate of the cardinality and diameter between consecutive MSTD sets in a sequence generated by this method is linear.

\begin{rem}
    We have not proven that there exists an $n$ for every MSTD set that satisfies the initial two conditions. However, empirical observations suggest that nearly every MSTD set does have such a valid $n$, and many have more than one. For the Conway set, we could use $n = 17,18$ or $20$ to generate the sequence.
    
    As to why this is, we note that both $x$ and $y$ increase at roughly the same rate as $n$ increases for most MSTD sets, so the second condition is likely to be met for larger $n$. On the other hand, we reason that the first condition is more likely to be met for smaller $n$. This is because values greater than $n$ in $A_1+A_1$ and values less than $0$ in $A_1-A_1$ are more likely to fill in the gaps. In other words, for smaller $n$, there is a higher probability that the number of unique elements modulo $n$ in both the sumset and difference set equals $n$ itself. We also note that meeting the first condition is impossible for $n > 2\cdot\max A_1$ because the number of unique sums and differences modulo $n$ are $\big|A_1+A_1\big|$ and $\big|A_1-A_1\big|$, respectively.
\end{rem}

\subsection{Non-Filling in Method 2}

We now construct a second method that generates the desired sequence and has $\big|A_{i+1}\big| = \big|A_i\big|+1$ for all $A_i$ in the sequence. Note that this is the smallest possible cardinality growth rate for any sequence.

We use the following results taken from Nathanson \cite{nathanson2006}.

\begin{thm} \label{NathansonTheorem1}
Let \( m, d, \) and \( k \) be integers such that \( m \geq 4 \), \( 1 \leq d \leq m - 1 \), \( d \ne m/2 \), and
\begin{itemize}
    \item \( k \geq 3 \) if \( d < m/2 \)
    \item \( k \geq 4 \) if \( d > m/2 \).
\end{itemize}
Define
\begin{align*}
B \ &= \ [0, m - 1] \setminus \{d\} \\
L \ &= \ \{m - d, 2m - d, \ldots, km - d\} \\
a^* \ &= \ (k + 1)m - 2d \\
A^* \ &= \ B \cup L \cup (a^* - B) \\
A \ &= \ A^* \cup \{m\}.
\end{align*}

Then \( A \) is a MSTD set of integers, with $2m \in (A+A) \setminus (A^*+A^*)$ and $A^*-\{m\} \subset A^*-A^*$.
\end{thm}

\begin{lem} \label{lemma1}
    Let $m \geq 4$ and $2\leq r  \leq m - 3$, with $m,r \in \mathbb{N}$. Let $B = [0,m-1] \setminus \{r\}$. Then $B+B  =  [0,2m-2]$, and $B-B  =  [-(m-1),m-1]$.
\end{lem}

Consider a MSTD set $A$ that is constructed in the way described in Theorem \ref{NathansonTheorem1}, with the additional constraints that \begin{align}
    \notag &m  \ \equiv \ 0 \bmod{4} \\ 
    &d \ \in \ \left\{\frac{m}{4},\frac{3m}{4}\right\}.
\end{align}

Now, let \begin{align}
\notag A_1 \ &= \ A \cup \{-d,(k+1)m-d\} \\
\notag A_1^* \ &= \ A^*  \cup \{-d,(k+1)m-d\} \\
L_1 \ &= \ L \cup \{-d,(k+1)m-d\}.
\end{align} 

We show that $A_1$ is MSTD. Since $-d$ and $(k+1)m -d$ are symmetric with respect to $a^*$, we have \begin{align}
    \big|A_1^* + A_1^*\big| \ = \ \big|A_1^* - A_1^*\big|.
\end{align}

We know $2m \notin A^* + A^*$, so showing $2m \notin A_1^* + A_1^*$ amounts to showing there exists no $a_1,a_2 \in A_1^*$ such that $a_1 -d = 2m$ or $a_2+ (k+1)m-d = 2m$. If such an $a_1$ existed, then $a_1 =  2m + d$. To prove that $2m+d \notin A_1^*$, we show individually that $2m+d$ is not in $B$, $L_1$, and $a^*-B$.
 \begin{itemize}
    \item Since $\max(B) < 2m+d$, we have that $2m+d \notin B$. 
    \item Suppose $2m+d \in L_1$. Then, since $d<m$, it follows that $2m+d = 3m-d$, which implies $d = m/2$. This is a contradiction since we assumed $d \ne m/2$.
    \item Assume $2m+d \in a^*-B$. Then, $2m+d =  (k+1)m-2d -b$ for some $b \in B$. Rearranging, we get $b = (k-1)m-3d$. To show $b \notin B$, we consider the possible values of $d$.
    
\begin{enumerate}
   \item If $d = m/4$, then \begin{align}
       b \ = \ \left(k-\frac{7}{4}\right)m.
   \end{align}
Since $k \geq 3$, this implies $b > m$, which is a contradiction.
\item If $d = 3m/4$, then \begin{align}
    b \ = \ \left(k-\frac{13}{4}\right)m.
\end{align}
If $k > 4$, then $b > m$. If $k = 4$, then $b = d$, which is a contradiction.
\end{enumerate}
In either case, $b \notin B$, which is a contradiction. Therefore, $2m+d \notin a^*-B$.
\end{itemize}
Thus, no possible $a_1 \in A_1^*$ exists.

If we solve for $a_2$ we get $a_2 = (1-k)m+d$. However, $a_2 \notin A_1^*$ because $a_2 < \min(A_1^*) = -d$. Thus, $2m \notin A_1^*$. 

Next, we prove that $(A_1+A_1) \setminus (A_1^*+A_1^*) = \{2m\}$ by showing $m+A_1^* \subseteq A_1^*+A_1^*$. 

\begin{itemize}
\item First, we show $m+B \subset A_1^*+A_1^*$. We have \begin{align}
    m+B \ = \ [m,2m-1] \setminus \{d+m\}.
\end{align}

Since $2m-d \in L_1$, and $d-1 \in B$, it follows that \begin{align}
  (2m-d)+(d-1) \ = \ 2m-1 \in A_1^*+A_1^*.
\end{align}
\begin{enumerate}
\item If $1 < d < m-2$, then by Lemma \ref{lemma1} \begin{align}
    B+B \ = \ [0,2m-2],
\end{align}
so $m+B \subset A_1^*+A_1^*$. 
\item If $d = 1$, then \begin{align}
    B+B \ = \ [0,2m-2] \setminus \{2\}
\end{align}
so $m+B \subset A_1^*+A_1^*$. 
\item If $d = m-2$, then \begin{align}
    B+B \ = \ [0,2m-2] \setminus \{2m-3\}.
\end{align}
Consider $2m-d = m+2 \in L$ and $m-5 \in B$. We know $m-5 \in B$ since $d = m-2$ implies $m \geq 5$ because $d \ne m/2$. Then \begin{align}
    (m+2)+(m-5) \ = \ 2m-3 \in A_1^*+A_1^*.
\end{align}
So $m+B \subset A_1^*+A_1^*$. 
\item If $d = m-1$, then \begin{align}
    B+B = [0,2m-4].
\end{align}
Consider $2m-d = m+1 \in L_1$, and $m-3$, $m-4 \in B$. Hence \begin{align}
    \notag (m+1) + (m-4) \ &= \ 2m-3 \ \in \ A_1^*+A_1^* \\
    (m+1) + (m-3) \ &= \ 2m-2 \ \in \ A_1^*+A_1^*.
\end{align}
So, $m+B \subset A_1^*+A_1^*$.
\end{enumerate}
Combining all cases, we conclude \begin{align}m+B \ \subset \  A_1^*+A_1^*. \label{m+B subset A_1*+A_1*}
\end{align}

\item Next, we show $m+L_1 \subset A_1^*+A_1^*$. Take $l \in L_1$ such that $l \ne  \max(L_1)$. Then, $l+m \in L_1$, and since $0 \in B$, we have $l+m \in A_1^*+A_1^*$. If $l = \max(L_1) = (k+1)m-d$, then \begin{align}
    l + m \ = \ (k+2)m-d.
\end{align} 
\begin{enumerate}
\item If $d = m/4$, then \begin{align}
    \notag l+m \ &= \ \left(k+\frac{7}{4}\right)m \\
    \notag 2m - d \ &= \ \frac{7m}{4} \\
    a^*-\frac{m}{2} \ &= \ km. \label{a*-m/2v1}
\end{align}
Clearly, $7m/4 \in L_1$, and since $d \ne m/2$, we have $km \in a^*-B$. Thus \begin{align}
    l+m \ &\in \ A_1^*+A_1^*.
\end{align}
\item If $d = 3m/4$, then \begin{align}
    \notag l+m \ &= \ \left(k+\frac{5}{4}\right)m \\
    \notag 3m - d \ &= \ \frac{9m}{4} \\
    a^*-\frac{m}{2} \ &= \ (k-1)m.\label{a^*-m/2v2}
\end{align}
Clearly, $9m/4 \in L_1$, and $(k-1)m \in a^*-B$. Therefore \begin{align}
    l+m \ &\in \ A_1^*+A_1^*.
\end{align}
\end{enumerate}
Thus \begin{align}
    m+L_1 \ \subset \ A_1^*+A_1^*. \label{m+L_1 subset A_1*+A_1*}
\end{align}

\item Lastly, we show $m+a^*-B \subset A_1^*+A_1^*$. First \begin{align}
    x \ \in \ m+ a^*-B & \ \implies \ x \ = \ a^*+(m-b), \ b \in B.
\end{align}
Clearly, $a^* \in a^*-B$. We show $m-b \in A_1^*$. For all $b \in B \setminus \{0\}$, we have \begin{align}
    m-B \setminus \{0\} \ = \ [1,m-1] \setminus \{m-d\}.
\end{align}

If $m-b  \ne d$, then $m-b \in B$ and $x \in A_1^*+A_1^*$. We consider two cases for when $m-b = d$. 
\begin{enumerate}
\item $d = m/4$. Then $m-b = d \implies b = 3m/4$ and \begin{align}
    x \ &= \ \left(k+\frac{3}{4}\right)m.
\end{align}
We observe that \begin{align}
    (k+1)m-d \ = \ \left(k+\frac{3}{4}\right)m,
\end{align}
so $x\in L_1$, and $x \in A_1^*+A_1^*$.

\item $d  = 3m/4$. Then $m-b = d \implies b = m/4$, and \begin{align}
    x \ &= \ \left(k+\frac{1}{4}\right)m.
\end{align}
We observe that \begin{align}
    (k+1)m-d \ = \ \left(k+\frac{1}{4}\right)m,
\end{align}
so $x\in L_1$, and $x \in A_1^*+A_1^*$. 
\end{enumerate}
If $b = 0$, then \begin{align}
    x \ = \ (k+2)m - 2d.
\end{align}
Taking $(k+1)m-d, \ m-d \in L_1$, we get \begin{align}
    (k+1)m-d+(m-d) \ = \ (k+2)m -2d.
\end{align}

Thus $x \in A_1^*+A_1^*$, and \begin{align}
    m+a^*-B \ \subset \ A_1^*+A_1^*. \label{m+a*-B subset A_1*+A_1*}
\end{align} 
\end{itemize}

From \eqref{m+B subset A_1*+A_1*}, \eqref{m+L_1 subset A_1*+A_1*}, and \eqref{m+a*-B subset A_1*+A_1*}, we conclude that \begin{align}
    m+A_1^* \subset  A_1^*+A_1^*.
\end{align}

Therefore, $(A_1+A_1) \setminus (A_1^*+A_1^*) = \{2m\}$, and \begin{align}
    \big|A_1+A_1 \big| \ = \ \big|A_1^*+A_1^* \big| + 1. \label{A_1sums+1are_equal}
\end{align}

Next, we show $\big|A_1-A_1\big| = \big|A_1^*-A_1^*\big|$ by showing $\{m\}-A_1^* \subseteq A_1^*-A_1^*$. In the proof of Theorem \ref{NathansonTheorem1}, Nathanson showed that $\{m\}-A^* \subseteq A^*-A^*$ \cite{nathanson2006}. It remains to be shown that $\{m\}-\{-d,(k+1)m+d\} \subseteq A_1^*-A_1^*$.

\begin{itemize}
   
\item Consider $m-(-d) = m+d$. 
\begin{enumerate}
\item If $d = m/4$, then $m-2d \in B$. Taking $2m-d \in L_1$, we have \begin{align}
    (2m-d) - (m-2d) \ = \ m+d
\end{align}
so $m+d \in A_1^*-A_1^*$. 

\item If $d = 3m/4$, then $m-d < m/2$, so $2m-2d < m$. This gives $2m-2d \in B$ since $2m-2d \ne d$. Taking $3m-d \in L_1$, we have
\begin{align}
    (3m-d) - (2m-2d) \ = \ m+d
\end{align}
so $m+d \in A_1^*-A_1^*$.
\end{enumerate}

\item $\left((k+1)m-d\right) - m = km-d \in A_1^* -A_1^*$, since $0 \in A_1^*$.
\end{itemize}
Thus, $m-A_1^* \subseteq A_1^*-A_1^*$, and since $A_1^* \subset A_1$, \begin{align}\big|A_1-A_1\big| \ = \ \big|A_1^*-A_1^*\big|.\label{A_1diffsareequal}
\end{align} 

\noindent From \eqref{A_1sums+1are_equal} and \eqref{A_1diffsareequal}, we have that $A_1$ is a MSTD set satisfying \begin{align}
    \big|A_1+A_1\big| \ = \ \big|A_1-A_1\big|+1.
\end{align}

For $r \geq 1$, assume that $A_{2r-1}$ is a MSTD set constructed as \begin{align}
   \notag L_{2r-1} \ &= \ \{-(r-1)m-d, \dots, (k+r-1)m -d, \ (k+r)m-d\} \\
    \notag A_{2r-1}^* \ &= \ B \ \cup \ L_{2r-1} \ \cup \ a^*-B \\
    A_{2r-1} \ &= \ A_{2r-1}^* \ \cup \ \{m\}. 
\end{align}
Assume further that $A_{2r-1}$ satisfies \begin{align}
    \notag m-A_{2r-1}^* \ &\subseteq \ A_{2r-1}^* -A_{2r-1}^* \\
    \notag 2m \ &\notin \ A_{2r-1}^* + A_{2r-1}^* \\
    \big|A_{2r-1}+A_{2r-1}\big| \ &= \ \big|A_{2r-1} - A_{2r-1}\big| + 1. \label{ZZzz}
\end{align}

We have already proven that $A_1$ satisfies \eqref{ZZzz}. Let \begin{align}
A_{2r+1} \ = \ A_{2r-1} \ \cup \ \{-rm-d, (k+r+1)m-d\}.
\end{align}

We show $A_{2r+1}$ is MSTD and satisfies \eqref{ZZzz}. First, we define \begin{align}
   \notag L_{2r+1} \ &= \ L_{2r-1} \ \cup \ \{-rm-d, (k+r+1)m-d\} \\
    \notag A_{2r+1}^* \ &= \ B \ \cup \  L_{2r+1} \ \cup \ a^*-B \\
    A_{2r+1} \ &= \ A_{2r+1}^* \cup \{m\}.  
\end{align}

Since $A_{2r+1}^*$ is symmetric with respect to $a^*$, we have that \begin{align}
    \big|A_{2r+1}^*+A_{2r+1}^*\big| \ = \ \big|A_{2r+1}^* - A_{2r+1}^*\big| \label{symmetric2r+1}.
\end{align}

We now show that $2m \notin A_{2r+1}^*+A_{2r+1}^*$. From our inductive assumption, $2m \notin A_{2r-1}^*+A_{2r-1}^*$, so we need to show that there exists no $a_1,a_2 \in A_{2r+1}^*$ such that $a_1 -rm-d = 2m$ or $a_2 +(k+r+1)m-d = 2m$. 

\begin{itemize}
\item Rearranging, we get $a_1 = (r+2)m+d$. We show $a_1 \notin A_{2r+1}^*$. Clearly, $a_1 \notin B$ since $m-1 < a_1$. Additionally, $d < m$ and $a_1 \in L_{2r+1}$ implies \begin{align}
 \notag (r+2)m+d \ &= \ (r+3)m -d, \\
 2d \ &= \ m,
\end{align}
which is impossible. Thus, $a_1 \notin L_{2r+1}$. 

If $a_1 \in a^*-B$, then $\exists b\in B$ such that \begin{align}
    \notag(r+2)m+d \ &= \ (k+1)m -2d - b \\
    b \ &= \ (k-r-1)m-3d. 
\end{align}
\begin{enumerate}
\item If $d = m/4$, then \begin{align}
    b \ = \ \left(k-r-\frac{7}{4}\right)m.
\end{align}
If $b \in B$, then \begin{align}
    \notag 0 \ &\leq \ \left(k-r-\frac{7}{4}\right)m \ < \ m \\
    0 \ &\leq \ k-r-\frac{7}{4} \ < \ 1 \label{eqn-k-r-7/4}.
\end{align}

Since $k,r \in\mathbb{Z}$, the only instance where \eqref{eqn-k-r-7/4} occurs is when $k-r = 2$. In this case, we have \begin{align}
    k-r-\frac{7}{4} \ &= \ \frac{1}{4} \ = \ d.
\end{align}

Since $d \notin B$, there exists no $b \in B$ such that $a_1 = a^*-b$.

\item If $d = 3m/4$, then \begin{align}
    b \ = \ \left(k-r-\frac{13}{4}\right)m.
\end{align}
If $b \in B$, then \begin{align}
   \notag0 \ &\leq \ \left(k-r-\frac{13}{4}\right)m \ < \ m \\
   0 \ &\leq \ k-r-\frac{13}{4} \ < \ 1 \label{Eqn-13/4}.
\end{align}

Since $k,r \in\mathbb{Z}$, the only instance where \eqref{Eqn-13/4} occurs is when $k-r = 4$. In this case, we have \begin{align}
    k-r-\frac{13}{4} \ &= \frac{3}{4} \ = \ d.
\end{align}

Since $d \notin B$, there exists no $b \in B$ such that $a_1 = a^*-b$.
\end{enumerate}
Combining all results together, we have that $a_1 \notin A_1^*$.

\item For $a_2$, rearranging gives $a_2 = (1-k-r)m +d$. However, $a_2 \notin A_1^*$ because $a_2 < -rm-d = \min(A_1^*)$. 
\end{itemize}
Thus, no possible $a_1,a_2$ exist, and we have that \begin{align}
  2m \notin A_{2r+1}^*+A_{2r+1}^*. 
\end{align} 

We now show that $(A_{2r+1}+A_{2r+1}) \setminus (A_{2r+1}^*+A_{2r+1}^*) = \{2m\}$. In the construction of $A_1$ we proved \eqref{m+B subset A_1*+A_1*} and \eqref{m+a*-B subset A_1*+A_1*}. Since $A_1^* \subset A_{2r+1}^*$, we have \begin{align}
    \notag m+B \ &\subset \ A_{2r+1}^*+A_{2r+1}^* \\
    m+a^*-B \ &\subset \ A_{2r+1}^*+A_{2r+1}^*. \label{m+a*-Bsusbetgen}
\end{align}

We prove that $m+ L_{2r+1} \subset A_{2r+1}^*+A_{2r+1}^*$. 
\begin{itemize}
\item For $l \in L_{2r+1}$ such that $l \ne \max(L_{2r+1})$, we have $l+m \in L_{2r+1}$. Thus, $l+m \in A_{2r+1}^*+A_{2r+1}^*$.

\item If $l = \max(L_{2r+1}) = (k+r+1)m-d$, then \begin{align}
    l+m \ = \ (k+r+2)m-d.
\end{align}
\begin{enumerate}
\item If $d = m/4$, then \begin{align}
    l+m \ = \ \left(k+r+\frac{7}{4}\right)m.
\end{align}
Since $k \geq 3$, it follows that $(r+2)m-d \in L_{2r+1}$. We have \begin{align}
    (r+2)m-d \ = \ \left(r+\frac{7}{4}\right)m. \label{DIFF_EQN1-7/4}
\end{align}

In the construction of $A_1$, we showed that $a^*-m/2 \in a^*-B$, and that \begin{align} 
    a^* - \frac{m}{2} \ &= \ km. \label{DIFF_EQN2-7/4}
\end{align}

Combining the two sums, we get  \begin{align}
    \left(r+\frac{7}{4}\right)m + km \ = \ l +m
\end{align}
so $l+m \in A_{2r+1}^*+A_{2r+1}^*$.

\item If $d = 3m/4$, then \begin{align}
    l+m \ = \ \left(k+r+\frac{5}{4}\right)m.
\end{align}
Since $k \geq 4$, it follows that $(r+3)m-d \in L_{2r+1}$. We have\begin{align}
    (r+3)m-d \ = \ \left(r+\frac{9}{4}\right)m.\label{DIFF_EQN1-5/4}
\end{align}

In the construction of $A_1$, we showed that $a^*-m/2 \in a^*-B$, and that \begin{align}
    a^* - \frac{m}{2} \ = \ (k-1)m. \label{DIFF_EQN2-5/4}
\end{align}

Combining the two sums, we get  \begin{align}
    \left(r+\frac{9}{4}\right)m + (k-1)m \ = \ l+m
\end{align}
so $l+m \in A_{2r+1}^*+A_{2r+1}^*$.
\end{enumerate}
\end{itemize}
Thus, \begin{align}
    m+ L_{2r+1} \subset A_{2r+1}^*+A_{2r+1}^*. \label{yolo}
\end{align}
Taking \eqref{m+a*-Bsusbetgen} and \eqref{yolo}, we get \begin{align}
    m+A_{2r+1}^* \ \subset \ A_{2r+1}^*+A_{2r+1}^*,
\end{align}
which yields \begin{align}
    |A_{2r+1}+A_{2r+1}| \ = \ |A_{2r+1}^*+A_{2r+1}^*|+1. \label{r_sum_cardinality_plus1}
\end{align}

Now we show that $m-A_{2r+1}^* \subset A_{2r+1}^*-A_{2r+1}^*$. In the proof of Theorem \ref{NathansonTheorem1}, Nathanson \cite{nathanson2006} showed \begin{align}
    \notag m-B \ &\subset \ A_{2r+1}^*-A_{2r+1}^* \\
    a^*-B-m \ &\subset \ A_{2r+1}^* - A_{2r+1}^* \label{NathansonShowed1}.
\end{align} 

We show that $L_{2r+1}-m \subset A_{2r+1}^*-A_{2r+1}^*$. 
\begin{itemize}
\item Take $l \in L_{2r+1}$ such that $l \ne \min(L_{2r+1})$. Then $l-m \in L_{2r+1}$, so $l-m \in A_{2r+1}^*-A_{2r+1}^*$.

\item Let $l = \min(L_{2r+1}) =  -rm-d$. Showing $l-m \in A_{2r+1}^*-A_{2r+1}^*$ is the same as showing $m-l \in A_{2r+1}^*-A_{2r+1}^*$. As such, \begin{align}
    m-l \ = \ (r+1)m+d.
\end{align}

\begin{enumerate}
\item If $d = m/4$, then $m-2d \in B$. Consider $(r+2)m-d \in L_{2r+1}$. Then \begin{align}
    \notag (r+2)m-d - (m-2d) \ = \ m-l& \\
    m-l \ \in \ A_{2r+1}^*-A_{2r+1}^*.&
\end{align}

\item If $d = 3m/4$, then $2m-2d \in B$. Consider $(r+3)m-d \in L_{2r+1}$. Then \begin{align}
    \notag (r+3)m-d - (2m-2d) \ = \ m-l& \\
    m-l \ \in \ A_{2r+1}^*-A_{2r+1}^*.&
\end{align}
\end{enumerate}
\end{itemize}
Thus \begin{align}
m - L_{2r+1} \ \subset \ A_{2r+1}^*-A_{2r+1}^*. \label{Mykonos}
\end{align}
Taking \eqref{NathansonShowed1} and \eqref{Mykonos} gives \begin{align}
    m-A_{2r+1}^* \ \subset \ A_{2r+1}^*-A_{2r+1}^*,
\end{align}
which yields 
\begin{align}
    \big|A_{2r+1}-A_{2r+1}\big| \ = \ \big|A_{2r+1}^* - A_{2r+1}^*\big|. \label{2r+1diffequals}
\end{align}

Combining \eqref{symmetric2r+1}, \eqref{r_sum_cardinality_plus1}, and \eqref{2r+1diffequals} shows that $A_{2r+1}$ is a MSTD set which satisfies \eqref{ZZzz}.

Now, define $A_{2r} =  A_{2r-1} \cup \{(k+r+1)m-d\}$. We show that \begin{align}
  \big|(A_{2r}+A_{2r}) \setminus (A_{2r-1}+A_{2r-1})\big| \ \leq \ m. 
\end{align} 
\begin{itemize}
\item First, denote $p = (k+r+1)m-d$. Since $\max(A_{2r-1}) < p$, we have that $2p$ is a new sum.

\item Now, consider $p+L_{2r-1}$. Take $l \in L_{2r-1}$. 
\begin{enumerate}
\item If $l < \max(L_{2r-1})$, then $l+m \in L_{2r-1}$ and \begin{align}
    l + p \ &= \ (l+m)+ (p-m) \label{L_2r-1CALCfora*-b}. 
\end{align}
Clearly, $p-m \in L_{2r-1}$, so $l+p \in A_{2r-1}+A_{2r-1}$.

\item If $l = \max(L_{2r-1}) \ = \ (k+r)m-d$, then $l+p$ could be a new sum. 
\end{enumerate} 

\item Consider $m+p = (k+r+2)m-d$. 
\begin{enumerate}
\item If $d = m/4$, then \begin{align}
    m+p \ = \ \left(k+r+\frac{7}{4}\right)m.
\end{align}
Since $(r+2)m-d \in L_{2r-1}$, we apply the same sum used in \eqref{DIFF_EQN1-7/4} and \eqref{DIFF_EQN2-7/4} to show $m+p \in A_{2r-1}+A_{2r-1}$.

\item If $d = 3m/4$, then \begin{align}
    m+p \ = \ \left(k+r+\frac{5}{4}\right)m.
\end{align}
Since $(r+3)m-d \in L_{2r-1}$ we apply the same sum used in \eqref{DIFF_EQN1-5/4} and \eqref{DIFF_EQN2-5/4} to show  $m+p \in A_{2r-1}+A_{2r-1}$. 
\end{enumerate}
Thus \begin{align}
 m+p \ \notin \ A_{2r-1}+A_{2r-1}. \label{m+pNOTIN} 
\end{align}
\item Consider $p+B$. We observe that \begin{align}
    p+B \ = \ [p, p+m-1] \setminus \{p+d\}.
\end{align}
\begin{enumerate}
\item Let $d = m/4$. Then \begin{align}
    p +b \ = \ \left(k+r+\frac{3}{4}\right)m + b.
\end{align}
Let $x\in \mathbb{Z}$ such that $-r+1 \leq x \leq k+r$. Then $xm-d =(x-1/4)m{}\in L_{2r-1}$. We now show that there exists some $b_1 \in B$ such that \begin{align}
    \notag\left(x-\frac{1}{4}\right)m+a^*-b_1 \ &= \ p+b \\   
    \notag \left(x-\frac{1}{4}\right)m+\left(k+\frac{1}{2}\right)m - b_1 \ &= \ \left(k+r+\frac{3}{4}\right)m+b \\
    \left(x-r-\frac{1}{2}\right)m-b \ &= \ b_1 \label{x-r-1/2v1}.
\end{align}
We know $0 \leq b, \ b_1 <m$. 
\begin{itemize}
\item[$\circ$] Let $x = r+1$. Then \begin{align}
    \notag b_1 \ &= \ \frac{m}{2}-b \\
    \notag 0 \ &\leq \ \frac{m}{2}-b \ < \ m \\
    b \ &\leq \ \frac{m}{2} \ < m+b.
\end{align}
Furthermore, we know $m/2-b \ne d$ since $b \ne d$. Thus, for all $b \leq m/2$, choosing $x = r+1$ and $b_1 = m/2-b$ satisfies \eqref{x-r-1/2v1}. 
\item[$\circ$] Let $x = r+2$. Then \begin{align}
   \notag  b_1 \ &= \ \frac{3m}{2}-b \\
    \notag0 \ &\leq \ \frac{3m}{2}-b \ < \ m \\
    -m+b \ &\leq \ \frac{m}{2} \ < \ b.
\end{align}
Furthermore, we know $3m/2-b \ne d$ since $b < m$. Thus, for all $b > m/2$, choosing $x = r+2$ and $b_1 = 3m/2-b$ satisfies \eqref{x-r-1/2v1}.
\end{itemize}
Therefore, since there is a valid $x$ and $b_1$ for every $b \in B$, we have $p+B\subset A_{2r-1}+A_{2r-1}$. 

\item Now, let $d = 3m/4$. Then \begin{align}
    p + b \ = \ \left(k+r+\frac{1}{4}\right)m + b.
\end{align}
Let $x\in \mathbb{Z}$ such that $-r+1 \leq x \leq k+r$. Then $xm-d = (x-3/4)m{}\in L_{2r-1}$. We now show that there exists some $b_1 \in B$ such that \begin{align}
    \notag\left(x-\frac{3}{4}\right)m+a^*-b_1 \ &= \ \left(k+r+\frac{1}{4}\right)m+b \\   
    \notag\left(x-\frac{1}{4}\right)m+\left(k-\frac{1}{2}\right)m - b_1 \ &= \ \left(k+r+\frac{3}{4}\right)m+b \\
    \left(x-r-\frac{3}{2}\right)m-b \ &= \ b_1 \label{x-r-1/2v2}.
    \end{align} 
We know $0 \leq b, \ b_1 < m$. 
\begin{itemize}
\item[$\circ$] Let $x = r+2$. Then \begin{align}
    \notag b_1 \ &= \ \frac{m}{2}-b \\
    \notag 0 \ &\leq \ \frac{m}{2}-b \ < \ m \\
    b \ &\leq \ \frac{m}{2} \ < \ m+b.
\end{align}
Furthermore, we know $m/2-b \ne d$ since $b \geq 0$. Thus, for all $b \leq m/2$, choosing $x = r+2$ and $b_1 = m/2 -b$ satisfies \eqref{x-r-1/2v2}. 
\item[$\circ$] Now, let $x = r+3$. Then \begin{align}
    \notag b_1 \ &= \ \frac{3m}{2}-b \\
    \notag 0 \ &\leq \ \frac{3m}{2}-b \ < \ m \\
    -m +b  &\leq \ \frac{m}{2} \ < \ b.
\end{align}
Furthermore, we know $3m/2-b \ne d$ since $b \ne d$. Thus, for all $b > m/2$, choosing $x = r+3$ and $b_1 = 3m/2 -b$ satisfies \eqref{x-r-1/2v2}.
\end{itemize}
\end{enumerate}

Therefore, since there is a valid $x$ and $b_1$ for every $d$ and $b \in B$, we have \begin{align}
p+B\ \subset \ A_{2r-1}+A_{2r-1}. \label{p+B}  
\end{align} 

Before considering $p+a^*-B$, we note that we currently have up to 2 potential new sums ($2p$ and $p+(k+r)m-d$).

\item We first observe that $p+a^*-B$ has exactly $m-1$ elements. We show that $\exists l \in L_{2r-1}$ such that $l \ne \max(L_{2r-1})$, $l \in a^*-B$. Recall that we showed $l+p \in A_{2r-1}+A_{2r-1}$ for this $l$ in \eqref{L_2r-1CALCfora*-b}.

Since elements of $L_{2r-1}$ are separated by exactly $m$ and since \begin{align}
    \notag a^*-B \ &= \ [km-2d+1,(k+1)m-2d] \setminus \{(k+1)m-3d\} \\
    \min(L_{2r-1}) \ &< \ \min(a^*-B) \ < \ \max(a^*-B) \ < \ \max(L_{2r-1}),
\end{align}
there must exist some $l \in L_{2r-1}$ such that
\begin{align}
    \notag l \ &\ne \ (k+r)m-d \\
    l \ &\in \ [km-2d+1,(k+1)m-2d].
\end{align}

We take this $l$ and let $l = xm-d$ for some $x \in \mathbb{Z}$, $-r+1 < x < k+r$. If $l = (k+1)m-3d$, then \begin{align}
    \notag xm -d \ &= \ (k+1)m-3d \\
    2d \ &= \ (k-x+1)m. \label{2d=k-x+1m}
\end{align}

However, regardless of whether $d = m/4$ or $3m/4$, we have that $2d\ \not\equiv \ 0 \bmod{m}$. Thus, \eqref{2d=k-x+1m} is impossible, and so $l \ne (k+1)m-3d$, implying \begin{align}
    l \ \in \ a^*-B.
\end{align}

Since $l+p \in A_{2r-1}+A_{2r-1}$, then although there are $m-1$ elements in $p+a^*-b$, there are only up to $m-2$ total possible new sums. 
\end{itemize}
Thus, the maximum number of new sums in $A_{2r}+A_{2r}$ that are not in $A_{2r-1}+A_{2r-1}$ is $(m-2)+2 = m$. Alternatively, \begin{align}
   |A_{2r}+A_{2r}| \ \leq \ |A_{2r-1}+A_{2r-1}| + m. \label{A_2rtotalsumsleqm} 
\end{align}

Now, consider $A_{2r}-A_{2r}$. First, we note that \begin{align}
    A_{2r-1}-A_{2r-1} \ \subset \ [-\big((k+r)m-d\big),(k+r)m-d].
\end{align}

Considering $p-B$, we have \begin{align}
    p-B \ = \ \left[(k+r)m-d+1,(k+r+1)m-d\right]\ \setminus \{(k+r+1)m-2d\}.
\end{align}

There are exactly $m-1$ elements in $p-B$, and since none of them lie in the range of $A_{2r-1}-A_{2r-1}$, we have that $p-B$ gives a set of $m-1$ differences in $A_{2r}-A_{2r}$ that are not in $A_{2r-1}-A_{2r-1}$. Similarly, $B-p$ contributes another $m-1$ new differences. Adding these up, we have that $A_{2r}$ has at least $2m-2$ new differences, i.e., \begin{align}
    \big|A_{2r}-A_{2r}\big| \ \geq \ \big|A_{2r-1}+A_{2r-1}\big|+2m-2 \label{A2rdiffs2m-2}.
\end{align}

Now, if we take the results from \eqref{ZZzz}, \eqref{A_2rtotalsumsleqm}, and \eqref{A2rdiffs2m-2}, we have \begin{align}
    \notag \big|A_{2r}+A_{2r}\big| \ \leq \ \big|A_{2r-1}+A_{2r-1}\big| + m \ &= \ |A_{2r-1}-A_{2r-1}|+m+1 \\
    \big|A_{2r-1}-A_{2r-1}\big|+m+1 \ &\leq \ \big|A_{2r}-A_{2r}\big|-m+3.
\end{align}
Thus\begin{align}
    \big|A_{2r}+A_{2r}\big| \ \leq \ \big|A_{2r}-A_{2r}\big|-m+3,
\end{align}
but since $m \geq 4$, we have that \begin{align}
    \big|A_{2r}+A_{2r}\big| \ < \ \big|A_{2r}-A_{2r}\big|.
\end{align}

Thus, $A_{2r}$ is a difference-dominated set. Furthermore, from our constructions of $A_{2r}$ and $A_{2r+1}$, we observe that \begin{align}
    A_{2r-1} \subset A_{2r} \subset A_{2r+1}.
\end{align} 
Since the construction of $A_{2r}$ and $A_{2r+1}$ works for all $r \geq 1$, and since $A_1$ satisfies the necessary conditions for $A_{2r-1}$, we are thus able to generate an infinite chain of sets \begin{align}
    A_1 \subset A_2 \subset A_3 \subset \cdots
\end{align}
such that $A_{2r-1}$ is MSTD and $A_{2r}$ is MDTS for all $r\in\mathbb{N}$. 

We summarize the results of this method in Theorem \ref{NFIMethod3Theorem}.

\begin{thm} \label{NFIMethod3Theorem}
   Let $A$ be a MSTD set that is constructed in the way described in Theorem \ref{NathansonTheorem1}, with the additional constraints that \begin{align*}
    &m  \ \equiv \ 0 \bmod{4} \\ 
    &d \ \in \ \left\{\frac{m}{4},\frac{3m}{4}\right\}.
\end{align*} 
Then $A_1 = A \cup \{-d, (k+1)m-d\}$ is MSTD. For $r \geq 1$, define \begin{align*}
    A_{2r} \ &\coloneqq \ A_{2r-1} \ \cup \ \{(k+r+1)m-d\} \\
    A_{2r+1} \ &\coloneqq \ A_{2r-1} \ \cup \ \{-rm-d, (k+r+1)m-d\}.
\end{align*}
Then $A_{2r}$ is MDTS, $A_{2r+1}$ is MSTD, and $A_1 \subset \cdots \subset A_{2r-1} \subset A_{2r} \subset A_{2r+1} \subset \cdots$ forms the desired alternating sequence. Additionally, if $A_i^* = A_i \setminus \{m\}$, we have \begin{enumerate}
    \item $m-A_{2r-1}^* \subset A_{2r-1}^*-A_{2r-1}^*$
    \item $m+A_{2r-1}^* \subset A_{2r-1}^*+A_{2r-1}^*$
    \item $\big|A_{2r+1}+A_{2r+1}\big| = \big|A_{2r-1}+A_{2r-1}\big|+ 1$ 
    \item $\big|A_{2r}+A_{2r}\big| \leq \big|A_{2r}-A_{2r}\big|-m+3$.
\end{enumerate}

\end{thm}


For an example, we apply Non-Filling in Method 2 to $A = \{0,2,3,4,7,11,12,14\}$, the Conway set, and we get 
\begin{align}
    \notag A_1 \ &= \ \{-1,0,2,3,4,7,11,12,14,15\} \\
    \notag A_2 \ &= \ A_1 \ \cup \ \{19\} \\
    A_3 \ &= \ A_2 \ \cup \ \{-5\},
\end{align}
and so on. Table \ref{NFIM3Tabler} summarizes the growth characteristics of this sequence.

\begin{table}[h] 
\centering
\caption{Non-Filling in Method 2 Example Sequence} 
\label{NFIM3Tabler} 
\begin{tabular}{|c|c|c|c|c|c|c|c|}
\hline
Set & $\big|A_i+A_i\big|$ & $\big|A_i-A_i\big|$ & Cardinality & Diameter & $\big|A_i\big|/\big|A_{i-1}\big|$ & D($A_i$)/D($A_{i-1}$)/& Density \\
\hline
$A_1$ & 32 & 31 & 10 & 16 & N/A & N/A & 0.625 \\
$A_2$ & 36 & 37 & 11 & 20 & 1.100 & 1.250 & 0.550 \\
$A_3$ & 40 & 39 & 12 & 24 & 1.091 & 1.200 & 0.500 \\
$A_4$ & 44 & 45 & 13 & 28 & 1.083 & 1.167 & 0.464 \\
$A_5$ & 48 & 47 & 14 & 32 & 1.077 & 1.143 & 0.437 \\
$A_6$ & 52 & 53 & 15 & 36 & 1.071 & 1.125 & 0.416 \\
$A_7$ & 56 & 55 & 16 & 40 & 1.067 & 1.111 & 0.400 \\
\vdots & \vdots &\vdots &\vdots & \vdots & \vdots & \vdots & \vdots \\
\hline
\end{tabular}\\
\raggedleft Limiting MSTD density: $0.250$\hspace{0.8cm}
\end{table}

As can be seen, Non-Filling in Method 2 gives a linear cardinality growth rate of 1 and a linear diameter growth rate of $m = 4$ between consecutive sets in the chain.

\begin{rem}
The reason we start the chain with $A_1$ instead of $A$ is that if if $d = m/4$, then $A \cup \{(k+1)m-d\}$ is not guaranteed to be MDTS. For instance, if $A$ is the Conway set, then $A \cup \{15\}$ is actually sum-difference balanced. Additionally, since $(k+1)m-d > \max(A)$ and $-d < \min(A)$, the `no filling in' constraint is met.

\end{rem}

\subsection{Non-Filling in Method 3} \label{NFI3Section}

Building off of Non-Filling in Method 2, we now construct a more specific method that has a smaller diameter growth rate between consecutive MSTD sets. For $k \ge 0$, let
\begin{align}
    \notag A_{4k+1}\ &=\ \{-5k-24,-5k-23,-5k-19,-5k-18,\dots,-14,-13\}\\
    \notag&\quad\cup\ \{21,22,26,27,\dots,5k+31,5k+32\}\\
    \notag&\quad\cup\ \{-8,-7,-4,-3,0,5,8,11,12,15,16\}\\
    &=\ \{-7,0,5,8,15\}\cup(5\cdot[-k-5,k+6]+\{1,2\})\setminus\{-9,1,2,6,7,17\}.
\end{align}

We show $A_{4k+1}$ is MSTD. Using the shorthand
\begin{equation} \label{Bn}
    B_n\ = \ A_n\setminus\{5\},
\end{equation}
we see that $B_{4k+1}$ is symmetric with respect to $8$, so
\begin{equation*}
    \big|B_{4k+1}+B_{4k+1}\big| \ = \ \big|B_{4k+1}-B_{4k+1}\big|.
\end{equation*}
Note that $B_{4k+1}+B_{4k+1}$ does not include $10$, as $B_{4k+1}$ does not contain any of $\{-5,2,10,17\}$, and the remaining elements of $B_{4k+1}$ are congruent to $1$ or $2$ modulo $5$. Since $5\in A_{4k+1}$, $A_{4k+1}+A_{4k+1}$ does include $10$. On the other hand, if we let $l\in[-k-5,k+5]$, $\delta\in\{1,2\}$ we form
\begin{align} \label{3.1sums}
    \notag(5l+\delta)+5 \ &= \ (5(l+1)+\delta)+0\\
    \notag(5(k+6)+\delta)+5 \ &= \ (5(k+4)+\delta)+15\\
    \notag-7+5 \ &= \ 11+(-13)\\
    \notag0+5 \ &= \ 8+(-3)\\
    \notag8+5 \ &= \ 16+(-3)\\
    15+5 \ &= \ 8+12,
\end{align} 
so $\big|A_{4k+1}+A_{4k+1}\big| = \big|B_{4k+1}+B_{4k+1}\big|+1$.

To see that $\big|A_{4k+1}-A_{4k+1}\big| = \big|B_{4k+1}-B_{4k+1}\big|$, we let $l\in[-k-4,k+6]$, $\delta\in\{1,2\}$ and form
\begin{align} \label{3.1diffs}
    \notag(5l+\delta)-5 \ &= \ (5(l-1)+\delta)+0\\
    \notag(5(-k-5)+\delta)-5 \ &= \ (5(-k-3)+\delta)-15\\
    \notag-7-5 \ &= \ 0-12\\
    \notag0-5 \ &= \ 11-16\\
    \notag8-5 \ &= \ 11-8\\
    15-5 \ &= \ 22-12.
\end{align}
Thus
\begin{equation}
    \big|A_{4k+1}+A_{4k+1}\big| \ = \ \big|A_{4k+1}-A_{4k+1}\big|+1. \label{A_4k+1}
\end{equation}

Now, let
\begin{align}
    \notag A_{4k+3} \ &= \ \{-5k-28,-5k-24,-5k-19,-5k-18,\dots,-14,-13\}\\
    \notag &\quad\cup \ \{21,22,26,27,\dots,5k+32,5k+36\}\\
    \notag&\quad\cup\ \{-8,-7,-4,-3,0,5,8,11,12,15,16\}\\
    \notag&=\ \{-5k-28,-7,0,5,8,15,5k+36\}\\
    \notag&\quad\cup\ (5*[-k-5,k+6]+\{1,2\})\setminus\{-9,6,7,17\}\\
    &=\ A_{4k+1}\cup\{-5k-28,5k+36\}.
\end{align}
As $B_{4k+3}$ is symmetric with respect to $8$, we have
\begin{equation*}
    \big|B_{4k+3}+B_{4k+3}\big| \ = \ \big|B_{4k+3}-B_{4k+3}\big|.
\end{equation*}
Note $B_{4k+3}+B_{4k+3}$ does not contain $10$ as $B_{4k+3}=B_{4k+1}\cup\{-5k-28,5k+36\}$ and $5k+33, -5k-31 \notin B_{4k+3}$. However, $A_{4k+3}+A_{4k+3}$ contains 10. Using the sums in \eqref{3.1sums} and
\begin{align*}
    (-5k-28)+5 \ &= \ (-5k-23)+0\\
    (5k+36)+5 \ &= \ (5k+26)+15,
\end{align*}
we see that $\big|A_{4k+3}+A_{4k+3}\big| = \big|B_{4k+3}+B_{4k+3}\big|+1$.

To show that $\big|A_{4k+3}-A_{4k+3}\big| = \big|B_{4k+3}-B_{4k+3}\big|$, we use the differences in \eqref{3.1diffs} and
\begin{align*}
    (-5k-28)-5 \ &= \ (-5k-18)-15\\
    (5k+36)-5 \ &= \ (5k+31)-0.
\end{align*}
Thus
\begin{equation}
    \big|A_{4k+3}+A_{4k+3}\big| \ = \ \big|A_{4k+3}-A_{4k+3}\big|+1. \label{A_4k+3}
\end{equation}

Let
\begin{equation}
    A_{4k+2} \ = \ A_{4k+1}\cup\{5k+36\}.
\end{equation}
For $l\in[-k-5,k+5]$, $\delta\in\{1,2\}$,
\begin{align}
    \notag(5l+\delta)+(5k+36)\ &=\ (5(l+1)+\delta)+(5k+31)\\
    \notag-7+(5k+36)\ &=\ (5k+32)+(-3)\\
    \notag0+(5k+36)\ &=\ (5k+31)+5\\
    \notag5+(5k+36)\ &=\ (5k+26)+15\\
    8+(5k+36)\ &=\ (5k+32)+12.
\end{align}
Since $\max(A_{4k+1}+A_{4k+1}) = 10k+64$ and $15$ is the largest element of $A_{4k+1}$ congruent to $0$ mod $5$, the new sums are
\begin{align}
    \notag(5k+31)+(5k+36)\ &=\ 10k+67\\
    \notag(5k+32)+(5k+36)\ &=\ 10k+68\\
    \notag 2(5k+36)\ &=\ 10k+72\\
    15+(5k+36)\ &=\ 5k+51.
\end{align}
Thus, there are $4$ new sums.

For $l\in[-k-4,k+6]$, $\delta\in\{1,2\}$,
\begin{align}
    \notag(5k+36)-(5l+\delta)\ &=\ (5k+31)-(5(l-1)+\delta)\\
    \notag(5k+36)-0\ &=\ (5k+32)-(-4)\\
    \notag(5k+36)-5\ &=\ (5k+31)-0\\
    \notag(5k+36)-8\ &=\ (5k+21)-(-7)\\
    (5k+36)-15\ &=\ (5k+21)-0.
\end{align}
Since $\max(A_{4k+1}-A_{4k+1}) = 10k+56$ and $-7$ is the smallest element of $A_{4k+1}$ congruent to $3$ mod $5$, the new positive differences are
\begin{align}
    \notag(5k+36)-(-5k-23)\ &=\ 10k+59\\
    \notag(5k+36)-(-5k-24)\ &=\ 10k+60\\
    (5k+36)-(-7)\ &=\ 5k+43.
\end{align}
Thus, there are $6$ new differences. Combining the amount of new sums and differences with \eqref{A_4k+1}, we get
\begin{align}
    \big|A_{4k+2}-A_{4k+2}\big| \ = \ \big|A_{4k+2}+A_{4k+2}\big|+1.
\end{align}

Let
\begin{equation}
    A_{4k+4} \ = \ A_{4k+3}\cup\{5k+37\}.
\end{equation}
For $l\in[-k-5,k+5]$, $\delta\in\{1,2\}$,
\begin{align}
    \notag(5l+\delta)+(5k+37)\ &=\ (5(l+1)+\delta)+(5k+32)\\
    \notag-7+(5k+37)\ &=\ (5k+22)+8\\
    \notag0+(5k+37)\ &=\ (5k+32)+5\\
    \notag5+(5k+37)\ &=\ (5k+27)+15\\
    15+(5k+37)\ &=\ (5k+36)+16.
\end{align}
Since $\max(A_{4k+3}+A_{4k+3}) = 10k+72$, $(5k+33) \notin A_{4k+3}$, and $8$ is the largest element of $A_{4k+1}$ congruent to $3$ mod $5$, the new sums are
\begin{align*}
    (5k+32)+(5k+37)\ &=\ 10k+69\\
    (5k+36)+(5k+37)\ &=\ 10k+73\\
    2(5k+37)\ &=\ 10k+74\\
    8+(5k+37)\ &=\ 5k+45.
\end{align*}
Thus there are $4$ new sums. Next, for $l\in[-k-4,k+6]$, $\delta\in\{1,2\}$,
\begin{align}
    \notag(5k+37)-(5l+\delta)\ &=\ (5k+32)-(5(l-1)+\delta)\\
    \notag(5k+37)-(-7)\ &=\ (5k+36)-(-8)\\
    \notag(5k+37)-5\ &=\ (5k+32)-0\\
    \notag(5k+37)-8\ &=\ (5k+22)-(-7)\\
    (5k+37)-15\ &=\ (5k+22)-0.
\end{align}
Since $\max(A_{4k+3}-A_{4k+3}) = 10k+64$, $(-5k-25) \notin A_{4k+3}$, and $0$ is the smallest element of $A_{4k+3}$ congruent to $0$ mod $5$, the new positive differences are
\begin{align*}
    (5k+37)-(-5k-24)\ &=\ 10k+61\\
    (5k+37)-(-5k-28)\ &=\ 10k+65\\
    (5k+37)-0\ &=\ 5k+37.
\end{align*}
Thus there are $6$ new differences. Combining the amount of new sums and differences with \eqref{A_4k+3}, we get
\begin{align}
    \notag&\big|A_{4k+4}-A_{4k+4}\big| \ = \ \big|A_{4k+4}+A_{4k+4}\big| + 1.
\end{align}

Furthermore, we have 
\begin{equation}
    A_{4k+1}\subset A_{4k+2}\subset A_{4k+3}\subset A_{4k+4}\subset A_{4k+5},
\end{equation}
so the sequence $A_{1}\subset A_{2}\subset\cdots$ alternates between being MSTD and MDTS. The growth characteristics of this sequence are given in Table \ref{NFM3.1Table}. 

\begin{table}[h] 
\centering
\caption{Non-Filling in Method 3 Sequence}
\label{NFM3.1Table}
\begin{tabular}{|c|c|c|c|c|c|c|c|}
\hline
Set & $\big|A_i+A_i\big|$ & $\big|A_i-A_i\big|$ & Cardinality & Diameter & $\big|A_i\big|/\big|A_{i-1}\big|$ & D($A_i$)/D($A_{i-1}$)/& Density \\
\hline
$A_1$ & 98 & 97 & 23 & 56 & N/A & N/A & 0.410\\
$A_2$ & 102 & 103 & 24 & 60 & 1.043 & 1.071 & 0.400\\
$A_3$ & 106 & 105 & 25 & 64 & 1.042 & 1.067 & 0.391\\
$A_4$ & 110 & 111 & 26 & 65 & 1.040 & 1.016 & 0.400\\
$A_5$ & 114 & 113 & 27 & 66 & 1.038 & 1.015 & 0.409\\
$A_6$ & 118 & 119 & 28 & 70 & 1.037 & 1.061 & 0.400\\
$A_7$ & 122 & 121 & 29 & 74 & 1.036 & 1.057 & 0.392\\
$A_8$ & 126 & 127 & 30 & 75 & 1.034 & 1.014 & 0.400\\
$A_9$ & 130 & 129 & 31 & 76 & 1.033 & 1.013 & 0.408\\
\vdots & \vdots &\vdots &\vdots & \vdots & \vdots & \vdots & \vdots \\
\hline
\end{tabular}\\
\raggedleft Limiting MSTD density: $0.400$\hspace{0.8cm}
\end{table}

As seen in the table, each set in the sequence has one more element than the previous. The diameter alternates between growing by 2 and 8 between consecutive MSTD sets in this sequence, which gives an `average' diameter growth rate of $5$. This is the best diameter growth rate yet, at the expense of a rather large initial set.

\begin{rem}
    Other sequences were found that had even smaller `average' diameter growth rates than 5. For instance, the set $\{0,1,2,4,5,9,12,13,14\}$ can be extended to form a sequence with an `average' diameter growth rate of 4.8. However, the period of the diameter growth rates for this sequence was 5 as opposed to 2, so the more concise proof in Non-Filling in Method 3 was given instead. A similar method using a symmetric set with every integer odd modulo 7, excluding $\{-25,-23,-6,5,22,24\}$ and including $\{-10,7,9\}$, gives an `average' diameter growth rate of 14/3, cardinality growth rate of 2, and period of 3.
\end{rem}

\section{Growth Rates}
We give a table that compares the growth characteristics of the various methods we constructed. The first three rows of the table cover the methods in \cite{AltSetsPt1}, while the last three rows cover the methods presented in this paper. Note that the $\big|A_1\big|$ and $A_1$ Diam. columns give the smallest possible cardinality and diameter for a set $A_1$ which a method can be applied to. Additionally, the Card. Rate and Diam. Rate columns measure the growth rate between consecutive MSTD sets in the sequence generated by the minimal $A_1$ for that method. 

\begin{table}[h]
\centering
\textbf{\caption{Minimal Method Growth Characteristics}} \vspace{0.1cm}
\label{tab:method1} 
\begin{tabular}{|c|c|c|c|c|c|}
\hline

\textbf{Method} & $\mathbf{\big|A_1\big|}$ & $\mathbf{A_1}$ \textbf{Diam.} & \textbf{Card. Rate} & \textbf{Diam. Rate} & \textbf{Type} \\
\hline 

Filling in 1 & 8 & 14 & $>3\cdot\big|A_{2i-1}\big|$ & $>3\cdot$ Diam$(A_{2i-1})$ & Exponential\\
\hline 
Filling in 2 & 11 & 19 & 20 & 20 & Linear\\ 
\hline 

Non-Filling in & 11 & 18 & 4 & 8 & Linear\\
\hline

Method 1 & 8 & 14 & 8 & 17 & Linear\\
\hline
Method 2 & 8 & 14 & 2 & 8 & Linear\\
\hline
Method 3 & 23 & 56 & 2 & 5 (avg) & Linear\\
\hline


\end{tabular}
\end{table}

\section{Future Directions}
We conclude by briefly discussing some avenues for future research which have arisen in the course of this paper. Firstly, one might investigate the properties of the modulus $n$ used in Theorem \ref{NFIMethod1Theorem}, for instance, we have not been able to prove when such an $n$ exists for a given MSTD set. Also, the ideas used in Section \ref{NFI3Section} could be modified, possibly in a similar manner to the remark following \ref{NFI3Section}, to give a slowly growing set for all odd integers. A strategy for even integers could also be found; such a set would have a diameter growth rate of $4$ between consecutive MSTD sets.

One could also investigate how these methods behave in other algebraic settings. One natural direction is to study fixed endpoint constructions, where one generates the desired alternating chains inside a subset with fixed endpoints. For example, given a prime $p$, one may ask how long of an alternating chain of sum- and difference-dominated supersets can be obtained within the finite field $\mathbb{F}_p$. More generally, one might explore how these constructions behave in fields of prime power order, or more broadly, in groups that are not necessarily abelian. In this context especially, methods with slower, controlled growth rates are far more applicable, hence the emphasis on developing efficient constructions.

\newpage
\ \\
\printbibliography
\ \\
\end{document}